\documentclass{amsart}
\usepackage[utf8]{inputenc}
\usepackage{amsmath, amsthm, amssymb, amsfonts, mathtools, stmaryrd, latexsym, enumerate}

\pdfoutput=1 % pictures on arXiv

\usepackage[dvipsnames]{xcolor}
\usepackage{float}
\usepackage{enumitem}
\usepackage{tikz}
\usepackage{comment}
\usepackage[linesnumbered,lined,commentsnumbered]{algorithm2e}

\usepackage{listings}
\usepackage{graphicx}
\usepackage{caption}
\usepackage[all]{xypic}
\usepackage{verbatim}
\usepackage{subcaption}

\usepackage[colorlinks]{hyperref}
\hypersetup{
    colorlinks,
    linkcolor={BrickRed},
    citecolor={ForestGreen},
    urlcolor={NavyBlue}
}
\usepackage[nameinlink]{cleveref}

\usepackage[parfill]{parskip}
\usepackage[margin=1in]{geometry}

\theoremstyle{plain}
\newtheorem{theorem}{Theorem}
\numberwithin{theorem}{section}
\newtheorem{maintheorem}{Theorem}

\newtheorem{proposition}[theorem]{Proposition}
\newtheorem{lemma}[theorem]{Lemma}
\newtheorem{corollary}[theorem]{Corollary}
\newtheorem{definition}[theorem]{Definition}

\theoremstyle{definition}
\newtheorem{remark}[theorem]{Remark}
\newtheorem{example}[theorem]{Example}

\newtheorem*{question}{Question}
\newtheorem*{notation}{Notation}
\newtheorem*{claim}{Claim}

\newcommand{\bbC}{\mathbb{C}}
\newcommand{\bbG}{\mathbb{G}}
\newcommand{\bbP}{\mathbb{P}}

\newcommand{\bbZ}{\mathbb{Z}}

\newcommand{\Hrk}{\operatorname{Hrk}}
\newcommand{\Sym}{\mathrm{Sym}}

\usepackage{dsfont}

\title{Finiteness of Hadamard ranks}

\author{Dario Antolini}
\address{%
	Dario Antolini\newline
	Dipartimento di Matematica, Università di Trento\newline
	Email: \href{mailto:dario.antolini-1@unitn.it}{dario.antolini-1@unitn.it}
}

\author{Edoardo Ballico}
\address{%
	Edoardo Ballico\newline
	Dipartimento di Matematica, Università di Trento\newline
	Email: \href{mailto:edoardo.ballico@unitn.it}{edoardo.ballico@unitn.it}
}

\author{Alessandro Oneto}
\address{%
	Alessandro Oneto\newline
	Dipartimento di Matematica, Università di Genova\newline
	Email: \href{mailto:alessandro.oneto@unige.it}{alessandro.oneto@unige.it}
}

	
\begin{document}

\begin{abstract} 
    The Hadamard rank of a point with respect to a projective variety is, if it exists, the minimum number of points of the variety whose coordinate-wise product is the given point. We classify the projective varieties for which the Hadamard rank is finite for any point. As a by-product we obtain the finiteness of the Hadamard rank with respect to varieties of tensors, such as Grassmannians, Chow varieties, varieties of reducible forms and their secant varieties, complementing previous known results on secant varieties of Segre-Veronese varieties. We prove sharp upper bounds on the maximum Hadamard rank for certain families of algebraic varieties: this is a consequence of a result on the lower semi-continuity of the Hadamard rank for curves that do not contain points with at least two zero coordinates.
\end{abstract}

\maketitle

\section{Introduction}
In \cite{CMS10:GeometryRBM}, the authors introduced the definition of \textit{Hadamard product} of algebraic varieties to study statistical models known as \textit{Restricted Boltzmann Machines} by means of algebraic and tropical geometry. Such approach was continued and generalized in \cite{CTY10:Implicitization,montufar2015discrete,montufar2018restricted}. 

Tensors encoding the joint probabilities that lie on Discrete Restricted Boltzmann Machines can be expressed as entry-wise product of low-rank tensors. This interpretation has been considered in \cite{seigal2018mixtures,oneto2023hadamard}. A general geometric framework for such decompositions of tensors was formalized in \cite{hadamardranks} where the definition of \textit{Hadamard ranks} was introduced with respect to any algebraic variety. This notion can be regarded as a multiplicative version of the classical additive notion of rank with respect to any algebraic variety which is geometrically studied through secant varieties, see \cite{oneto2025ranks} for a recent survey.

\begin{notation}
    Let $V$ be a $(N+1)$-dimensional $\bbC$-vector space and fix a basis $\{e_0,\ldots,e_N\}$. For any vector $p = p_0e_0 +\cdots + p_Ne_N$, we denote by $(p_0:\cdots:p_N)\in\bbP V = \bbP^N$ the corresponding projective point.     
\end{notation}

\begin{definition}\label{def:hadamard_product}
    Let $p = (p_0:\cdots:p_N), \ q=(q_0:\cdots:q_N)\in\bbP V$. Their \emph{Hadamard product} is defined, if it exists, as the coordinate-wise product $p\star q = (p_0q_0:\cdots:p_Nq_N)$.
\end{definition}
Let $X \subset \bbP V$ be any irreducible projective variety.
\begin{definition}\label{def:Hadamard_rank}
    The \emph{Hadamard-$X$-rank} of a point $q \in \bbP V$ is defined as 
    \[
        \Hrk_X(q) = \min\{m \mid q = p_1 \star \cdots \star p_m, \ p_i \in X\},
    \]
    with the convention that $\Hrk_X(q) = \infty$ if the set of such decompositions is empty.
\end{definition}
The question that we address in this paper, which naturally arises from this definition, is the following.

\begin{question}\label{question}
    For which $X \subset \bbP V$ and $q \in \bbP V$ the Hadamard-$X$-rank of $q$ is finite? For which $X\subset \bbP V$ the Hadamard-$X$-rank is finite for every point of $\bbP V$?    
\end{question}

Having in mind applications to tensor decompositions, in \cite{hadamardranks}, the authors focused on toric varieties, such as Segre-Veronese varieties, and their secant varieties. Since toric varieties are Hadamard-idempotent, see e.g., \cite[Proposition 4.7]{FOW:MinkowskiHadamard}, the Hadamard-rank with respect to toric varieties is infinite for every point outside the variety: in particular, that is the case for Segre-Veronese varieties. At the same time, it was proven in \cite[Corollary 3.8]{hadamardranks} that the Hadamard-rank with respect to secant varieties of toric varieties is finite everywhere. In the case of secants of Segre-Veronese varieties, this translates to the following: for any $r \geq 2$, any (partially-symmetric) tensor can be expressed as entry-wise product of finitely many (partially-symmetric) tensors of rank at most $r$. 

In the literature of tensor decompositions, there are several other interesting algebraic varieties, for example, Grassmannians, Chow varieties and varieties of reducible forms. All these are not toric varieties and the results of \cite{hadamardranks} about finiteness of Hadamard ranks do not apply.

In \cite[Corollary 3.40]{hadamardranks}, it was shown that the \textit{generic} Hadamard-$X$-rank is finite if and only if the variety $X$ is \textit{concise} (i.e., it is not contained in any coordinate hyperplane) and it is not contained in a \textit{binomial hypersurface}. However, the finiteness of the  generic Hadamard-$X$-rank does not guarantee that the Hadamard-$X$-rank is finite for \textit{any} point, see \cite[Example 3.10]{hadamardranks}. 

\subsection*{Main results.}
In this paper, we give a complete classification of the irreducible varieties for which the Hadamard-$X$-rank is finite for \textit{any} point of the ambient space.

\begin{notation}
    We write $H_i = \{ x_i = 0 \}$ for the $i$th coordinate hyperplane in $\bbP^N$, $i \in \{0,\ldots,N\}$.
\end{notation}

\begin{maintheorem}
\label{thm:maximum_finite}
    The Hadamard-$X$-rank is finite for every point if and only if $X$ is strongly concise, namely, for any $i \in \{0,\ldots,N\}$, $(X \cap H_i) \not\subset \bigcup_{j\neq i} H_j$.
\end{maintheorem}

In \Cref{ssec:tensors}, we will use \Cref{thm:maximum_finite} to show the finiteness of the Hadamard rank with respect to important varieties of tensors in the literature, such as Grassmannians, Chow varieties, varieties of reducible forms and tangential varieties. These results complement \cite[Corollary 3.9]{hadamardranks} on finiteness of Hadamard ranks with respect to secant varieties of Segre-Veronese varieties, justifying and supporting a new research direction on the study of different Hadamard ranks of tensors and their properties. 

For example, once we know that the maximum Hadamard-$X$-rank is finite, it is natural to look for bounds. We prove that under certain geometric conditions the upper bound is given by the dimension of the ambient space. We also note that such upper bound is sharp (\Cref{example:sharp}).

\begin{maintheorem}
\label{thm:maximum_notmanyzeroes}
    Let $X \subset \bbP^N$ be an irreducible variety of dimension $n\leq N/2$ such that $X \cap \Delta_{N-n-1} = \emptyset$. Then, $\Hrk_X(p) \leq N$ for any $p\in \bbP^N$.
\end{maintheorem}
The latter result is a corollary of a result on the lower semi-continuity of the Hadamard rank with respect to algebraic curves that do not contain points with at least two coordinates equal to zero (\Cref{thm:curves_finite}): this generalizes an analogous result for straight lines from \cite{bocci2016hadamard}. In \Cref{example:border_smallerthan_rank}, we show that the condition about the containment of points with two zero-coordinates is needed: indeed, we provide an example of a conic in $\bbP^2$ containing a coordinate point and for which the Hadamard rank is not lower semi-continuous. Such example justifies the definition of \textit{Hadamard border ranks}, see \Cref{def:hadamard_borderrank}.

\subsection*{Acknowledgments} All authors are members of INdAM-GNSAGA. AO has been partially supported by the Italian Ministry of University and Research in the framework of the Call for Proposals for scrolling of final rankings of the PRIN 2022 call - Protocol no. 2022NBN7TL (“Applied Algebraic Geometry of Tensors”). 

\section{Hadamard products of algebraic varieties}
\label{sec:definitions}

We first recall the basic definitions on Hadamard products of algebraic varieties. For a general reference, we refer to the recent book \cite{bocci2024hadamard}.

The Hadamard product of points defined in \Cref{def:hadamard_product} induces a rational map $h\colon \bbP V \times \bbP V \dashrightarrow \bbP V$. This can be regarded as the composition of the Segre embedding $\bbP V \times \bbP V \hookrightarrow \bbP (V\otimes V)$ with the linear projection onto the diagonal coordinates. This allows us to define the Hadamard product of algebraic varieties.

\begin{definition}
    Given two algebraic varieties $X,Y \subset \bbP V$, their \emph{Hadamard product} is the Zariski closure of the image of $X \times Y$ through the Hadamard product map, i.e.,
    \begin{equation}
    \label{def:HadamardProduct}
        X \star Y := \overline{h(X \times Y)} = \overline{\{p \star q ~|~ p \in X, \ q \in Y, \ p \star q \text{ exists}\}.}
    \end{equation}
    Given an algebraic variety $X \subset \bbP V$, we define its \emph{Hadamard powers} iteratively as:
    \[
        X^{\star 0} := (1:\cdots:1) \quad \text{ and } \quad X^{\star s} := X \star X^{\star (s-1)}.
    \]
\end{definition}

\begin{remark}
\label{rmk:hadamardirreducible}
    If $X, Y \subset \bbP V$ are irreducible, then $X \star Y \subset \bbP V$ is irreducible, since it is (the closure of) the image of a rational map.
\end{remark}

The notion of Hadamard rank with respect to an algebraic variety $X \subset \bbP V$ is strictly related to the Hadamard powers of $X$. If we look at the variety of points with bounded Hadamard-$X$-rank
\[
    \eta_m(X) = \overline{\{p \in \bbP V \mid \Hrk_X(p)\leq m\}},
\]
then, 
\[
    \eta_m(X) = X \cup X^{\star 2} \cup \cdots \cup X^{\star m}.
\]
Note that if $(1:\cdots:1) \in X$, then $X^{\star i} \subset X^{\star (i+1)}$ for any $i$, and, in particular, $\eta_m(X)=X^{\star m}$. Anyway, we always have that $\dim X^{\star i} \leq \dim X^{\star (i+1)}$ and thus $\dim \eta_m(X) = \dim X^{\star m}$. This allows us to interpret the \emph{generic Hadamard-$X$-rank}, which is defined as
\[
    \Hrk_X^\circ = \min\{r \mid \text{there exists a Zariski dense $U\subset X$ such that }\Hrk_X(p) = r \text{ for all } p \in U\},
\]
in terms of Hadamard powers of $X$. Indeed, the generic point $p \in \bbP V$ has Hadamard rank $m$ with respect to $X$ if and only if $m$ is the minimum integer $k$ for which $X^{\star k}$ fills the ambient space, namely, 
    \[
        \Hrk_X^\circ = \min\{ m \ | \ \eta_m(X) = \bbP V\} = \min\{ m \ | \ X^{\star m} = \bbP V\},
    \]
    with the convention that $\Hrk_X^\circ = \infty$ if $X^{\star m}$ does not fill the ambient space for any $m$.

The generic Hadamard-$X$-rank is not always finite, as the following example shows.

\begin{example}[Conciseness]\label{rmk:concise}
    Let $H_i = \{x_i = 0\}$ be the $i$th coordinate hyperplane in $\bbP^N$, where $i = 0,\ldots, N$. Notice that if $X \subset H_i$ for some $i$, then $X^{\star m} \subset H_i^{\star m} = H_i$ for all $m\geq 1$, and in particular $\Hrk_X^\circ = \infty$. In other words, if $\Hrk_X^\circ < \infty$ then $X$ is \emph{concise}, i.e., $X$ is not contained in any coordinate hyperplane.
\end{example}
A more general obstruction to the finiteness of the generic Hadamard rank is being contained in an {\em Hadamard-idempotent variety}, i.e.\ a variety $Y \subset \bbP V$ such that $Y^{\star 2} = Y$.
\begin{example}
    An example of concise varieties that are Hadamard-idempotent are {\em embedded toric varieties}, i.e.\ varieties parametrized by monomial maps, or, equivalently, such that their defining ideal has a generating set made by polynomials of the form ${\bf x}^\alpha - {\bf x}^\beta$.\footnote{We use the compact notation ${\bf x}^\alpha = x_0^{\alpha_0} \cdots x_N^{\alpha_N}$ for ${\bf x} =(x_0,\ldots,x_N) \in \bbC^{N+1}$ and $\alpha = (\alpha_0,\ldots,\alpha_N) \in \bbZ_{\geq 0}^{N+1}$.}
\end{example}

However, that is not the only case. Indeed, it is possible to have that $\Hrk_X^\circ = \infty$ even if the Hadamard powers of $X$ are not all contained in some specific proper variety, as we see in the following example.

\begin{notation}
    In what follows, by a {\em binomial} we mean a polynomial of the form $a{\bf x}^{\alpha} + b {\bf x}^\beta$ for some $a, b \in \bbC \smallsetminus\{0\}$ and $\alpha, \beta \in \bbZ_{\geq 0}^{N+1}$. By a {\em binomial hypersurface} we mean an hypersurface defined by a binomial equation.
\end{notation}

\begin{example}
    Let $c \in \bbC$ be a non-zero complex number which is not a root of unity. For all integers $N\geq 2$ and $d \geq 1$, we consider the binomial hypersurface $X_{d,N} = \{  x_1^d - c x_0^{d-1}x_N = 0 \} \subset \bbP^N$. Then, $X_{d,N}^{\star m} = \{ x_1^d - c^m x_0^{d-1}x_N = 0 \}$, see \cite[Proposition 4.5]{bocci2022hadamard}. For all $m\geq 1$, let $p_m = (1:\cdots:1:c^{-m})$. Since $p_m = p_1^{\star m}$ and $p_1 \in X_{d,N}$, we deduce $\Hrk_{X_{d,N}} (p_m) \leq m$. Actually, equality holds: if, by contradiction, we had $\Hrk_{X_{d,N}} (p_m) = r < m$, then $p_m \in X_{d,N}^{\star r}$, and hence $1 - c^rc^{-m} = 0$. This is a contradiction since $r-m \neq 0$ and $c$ is not a root of unity. In particular, $\Hrk_{X_{d,N}}(p_m) = m$ for all $m$.
\end{example}

In \cite{hadamardranks}, it was shown using tropical geometry that concise irreducible Hadamard-idempotent varieties must be contained in a binomial hypersurface. As a corollary, the following characterization of the varieties $X \subset \bbP V$ for which the generic Hadamard-$X$-rank is finite was obtained.

\begin{theorem}[\cite{hadamardranks}]\label{thm:finiteness_generic}
    Let $X \subset \bbP V$ be an irreducible variety. The following are equivalent:
    \begin{enumerate}
        \item $\Hrk_X^\circ < \infty$, i.e.\ the Hadamard-$X$-rank of the generic point $p \in \bbP V$ is finite
        \item $X$ is concise and not contained in any binomial hypersurface
        \item the vanishing ideal of $X$ does not contain any variable and any binomial.
    \end{enumerate}
\end{theorem}

The goal of the next section is to prove a classification of the irreducible varieties for which the Hadamard-$X$-rank is finite for \textit{every} point, i.e.\ not only for the general one.

\section{Finiteness of the maximum rank}

Given an irreducible variety $X \subset \bbP V$, we denote the \emph{maximum Hadamard-$X$-rank} as
\[
    \Hrk^{\max}_X = \min\{r \mid \forall p\in \bbP V,\ \Hrk_X(p) \leq r\}
\]
with the convention that $\Hrk^{\max}_X = \infty$ if there exists at least a point $p \in \bbP V$ for which $\Hrk_X(p) = \infty$.

Note that even if we require the finiteness of the generic Hadamard-$X$-rank, it is not guaranteed that the maximum Hadamard-$X$-rank is finite, as shown in the following example. 
\begin{example}[\cite{hadamardranks}]\label{example:maximum_nonfinite}
    Consider the quadric curve $Q = \{ x_0x_1 + x_0x_2 + x_1x_2 = 0 \} \subset \bbP^2$. Then, $Q^{\star 2} = \bbP^2$ and so $\Hrk_Q^\circ = 2$. However, every point $(a:b:0)$ such that $ab \neq 0$ cannot be written as a product of points on $Q$. Hence, $\Hrk_Q((a:b:0)) = \infty$ for all $a,b \in \bbC \smallsetminus \{0\}$.
\end{example}
However, in the previous example, the problem arises for points with some coordinate equal to zero. This motivated to look at points {\em with all non-zero coordinates}, and the following bound, inspired by the analogous result \cite[Theorem 1]{blekherman2015maximum} for additive $X$-ranks, was obtained.

\begin{theorem}[\cite{hadamardranks}]\label{thm:finitess_torus}
    Let $X \subset \bbP^N$ be a concise irreducible variety not contained in any binomial hypersurface. Then, $\Hrk_X(p) \leq 2 \Hrk_X^\circ$ for all $p =(p_0:\cdots:p_N) \in \bbP^N$ such that $\prod_i p_i \neq 0$.
\end{theorem}

If we look again at \Cref{example:maximum_nonfinite}, we realize that the intersection of $Q$ with any coordinate line $H_0, H_1, H_2 \subset \bbP^2$ yields only coordinate points, i.e.\ points having always at least two coordinates equal to zero. Since we want to avoid such a situation, we give the following definition.

\begin{definition}
\label{def:strongly_concise}
    A variety $X \subset \bbP^N$ is called {\em strongly concise} if, for each $i = 0,\ldots,N$, the intersection $X \cap H_i$ is not contained in a union of coordinate hyperplanes different from $H_i$, i.e., $(X\cap H_i) \not\subset \bigcup_{j\neq i}H_j$. 
\end{definition}

\begin{remark}
\label{rmk:strong->generic_finite}
    Clearly, if $X$ is strongly concise then $X$ is concise, as defined in  \Cref{rmk:concise}. Moreover, notice that binomial hypersurfaces are not strongly concise: in particular, if $X$ is strongly concise, then the generic Hadamard-$X$-rank is finite by \Cref{thm:finiteness_generic}.
\end{remark}

\begin{remark}
\label{rem:strongly_concise_ideal}
    From an algebraic point of view, saying that $X$ is strongly concise is equivalent to say that the vanishing (radical) ideal of $X$, denoted by $I(X)$, does not contain a polynomial of the form $x_iF + {\bf x}^\alpha$ where ${\bf x}^{\alpha} \in \bbC[x_0,\ldots,\widehat{x_i},\ldots,x_N]$ is a monomial not involving $x_i$. Indeed, if such a polynomial is identically zero on $X$, then $(X \cap H_i) \subset \{{\bf x} \mid {\bf x}^\alpha = 0\} \subset \bigcup_{j\neq i}H_j$. Vice versa, assume that $X$ is not strongly concise. Then, for some $i \in \{0,\ldots,N\}$, $(X \cap H_i) \subset \bigcup_{j \neq i} H_j$, and so $I\Big(\bigcup_{j \neq i} H_j\Big) \subset I(X \cap H_i)$. The ideal on the left is generated by $\prod_{j\neq i} x_j$, while the ideal on the right is $\sqrt{I(X) + (x_i)}$. By the inclusion, there exists $r \geq 1$ such that $\prod_{j\neq i}x_j^r \in I(X) + (x_i)$, from which $\prod_{j\neq i}x_j^r = G - x_iF$ for some $G \in I(X)$. Hence, the ideal $I(X)$ contains the polynomial $G = \prod_{j\neq i}x_j^r + x_i F$.
\end{remark}

We now prove \Cref{thm:maximum_finite} by separately proving the two implications.

The necessary condition is a generalization of \Cref{thm:finitess_torus}.
\begin{proposition}
\label{prop:bound_zeroes}
    Let $X \subset \bbP^N$ be a strongly concise irreducible variety. Then, for all $p \in \bbP^N$, 
    \[
        \Hrk_X(p) \leq 2\Hrk_X^\circ + z(p),
    \]
    where $z(p)$ denotes the number of zero coordinates of $p$. In particular, $\Hrk_X^{\max}<\infty$.
\end{proposition}
\begin{proof}
    Let $p = (p_0 : \cdots :p_N) \in \bbP^N$ and write $z(p) = z$. Since $X$ is strongly concise, for all $i = 0,\ldots, N$, there exists $q_i \in X$ such that if $q_i = (q_{i0}:\cdots:q_{iN})$, then $q_{ii}= 0$ and $\prod_{j \neq i} q_{ij} \neq 0$. Let $i_1,\ldots,i_z$ be all the indices such that $p_{i_j} = 0$ and let $q = q_{i_1}\star \cdots \star q_{i_z}$. Since $q$ and $p$ have the same number of zero coordinates, there exists $p' \in \bbP^N$ with all non-zero coordinates such that $p = p' \star q$. Since $X$ is strongly concise, $\Hrk_X^\circ <\infty$ (\Cref{rmk:strong->generic_finite}) and by \Cref{thm:finitess_torus} we know that $\Hrk_X(p') \leq 2 \Hrk_X^\circ$. At the same time, by construction, $\Hrk_X(q) \leq z$. Hence, $\Hrk_X(p) \leq \Hrk_X(p') + \Hrk_X(q) \leq 2 \Hrk_X^\circ + z$ and we conclude. 
\end{proof}

Now, we prove the opposite direction which concludes the proof of \Cref{thm:maximum_finite}.

\begin{proposition}\label{prop:maximum=>strongly_concise}
    Let $X \subset \bbP^N$ be a variety and assume that the maximum Hadamard-$X$-rank is finite. Then, $X$ is strongly concise.
\end{proposition}
\begin{proof}
    Suppose by contradiction that $X$ is not strongly concise. Then, there exists $i = 0,\ldots, N$ such that $(X \cap H_i) \subset \bigcup_{k\neq i} H_k$. In other words, for any point $y = (y_0 : \cdots :y_N) \in X$ with $y_i = 0$, we have $y_k = 0$ for some other $k \neq i$. Let $p = (p_0 : \cdots : p_N) \in \bbP^N$ be any point in the ambient space with $p_i = 0$ and $\prod_{j \neq i} p_j \neq 0$. Since the maximum Hadamard-$X$-rank is finite, $p = q_1\star \cdots \star q_m$ for some $q_1,\ldots,q_m \in X$. Since $p_i = 0$, then, for some $j = 1,\ldots, m$, the point $q_j = (q_{j0}: \cdots :q_{jN}) \in X$ satisfies $q_{ji} = 0$. Since $q_j \in X$ and $(X\cap H_i) \subset \bigcup_{k \neq i} H_k$, we have that $q_{jk} = 0$ for some $k \neq i$. It follows that $p_k = 0$ for some $k \neq i$, a contradiction.  
\end{proof}

\begin{corollary}
    For any $X \subset \bbP V$ irreducible variety, the following are equivalent:
    \begin{enumerate}
        \item $\Hrk_X^{\max}< \infty$, i.e.\ the Hadamard-$X$-rank is finite for any point $p \in \bbP V$
        \item $X$ is strongly concise
        \item the vanishing ideal of $X$ does not contain a polynomial of the form $x_i F+ {\bf x}^{\alpha}$, where the monomial ${\bf x}^\alpha \in \bbC[x_0,\ldots,\widehat{x_i},\ldots,x_N]$ does not involve the variable $x_i$.
    \end{enumerate}
\end{corollary}
\begin{proof}
    The proof follows from \Cref{rem:strongly_concise_ideal}, \Cref{prop:bound_zeroes,prop:maximum=>strongly_concise}.
\end{proof}

\subsection{Varieties of tensors}\label{ssec:tensors}

In what follows, we show that some families of varieties relevant for tensor decompositions are strongly concise. By \Cref{thm:maximum_finite}, this will imply that the Hadamard decompositions associated to them will always exist for any point in the ambient space.

Recall that the case of secant varieties of Segre-Veronese varieties was already considered in \cite{hadamardranks}.

\subsubsection{Grassmannians} We begin with skew-symmetric tensor decompositions.

\begin{lemma}
\label{lem:grassmannians}
    The Grassmannian of $k$-dimensional linear subsets of $\bbC^n$ is strongly concise in its Plücker embedding $\bbG(k,n) \hookrightarrow \bbP \bigwedge^{k} \bbC^{n}$.
\end{lemma}
\begin{proof}
    It is well known that a point in $\bbG(k,n)$ is determined by the equivalence class of a full rank $k\times n$ matrix. The Plücker embedding takes any matrix in the equivalence class into its $k \times k$ maximal minors. Hence, the strongly conciseness of $\bbG(k,n)$ in this embedding is equivalent to the following statement: 
    \begin{center}
        There exists a $k\times n$ matrix with exactly one $k\times k$ minor equal to zero.
    \end{center}
    We prove the statement in the following equivalent geometric reformulation.
    \begin{claim}
    \label{claim:grassmannian}
        For $k\leq n$, there exist $\{v_1,\ldots,v_n\} \subset \bbC^k$ such that $\{v_1,\ldots,v_k\}$ do not span $\bbC^k$, but any other subset $\{v_{i_1},\ldots,v_{i_k}\}$ spans $\bbC^k$.
    \end{claim}
    \begin{proof}[Proof of Claim]
        For all $k\geq 1$, we prove the theorem by induction on $n\geq k+1$. By construction, the points $v_1,\ldots, v_k \in \bbC^k$ span an hyperplane $H \cong \bbC^{k-1}$. In the base case $n = k+1$, we can just choose $v_{k+1}$ to be not in $H$. For the inductive step, we can choose $v_1,\ldots,v_{n-1} \in \bbC^k$ with the desired property. For all subsets $\{ i_1,\ldots, i_{k-1}\} \subset \{ 1,\ldots, n-1\}$, denote by $H_{i_1,\ldots,i_{k-1}} \subset \bbC^k$ the hyperplane spanned by $\{v_{i_1},\ldots, v_{i_{k-1}}\}$. Since $\bigcup_{i_1,\ldots,i_{k-1}} H_{i_1,\ldots,i_{k-1}} \neq \bbC^k$, we can choose $v_n \not \in H_{i_1,\ldots,i_{k-1}}$ for all $i_1,\ldots,i_{k-1}$. Then, $\{v_1,\ldots,v_{n-1},v_n\} \subset \bbC^k$ have the desired property.
    \end{proof}
    This concludes the proof of the lemma.
\end{proof}

We now turn to the case of symmetric tensors (or, homogeneous polynomials) with particular structures.

\subsubsection{Tangential variety to Veronese varieties}
\begin{definition}
    For all $d,n \geq 1$, the \emph{tangential variety} to the Veronese embedding of $\bbP^n$ in degree $d$ is:
    \[
        \tau_{d,n} = \{ [L^{d-1}M] : L, M \in (\bbC^{n+1})^\vee \} \subset \bbP\Sym^d\bbC^{n+1}.
    \]
    This is the Zariski closure of the union of all tangent spaces to the embedding $\bbP\bbC^{n+1} \hookrightarrow \bbP\Sym^d\bbC^{n+1}$.
\end{definition}

\begin{lemma}
\label{lem:tangential}
    The tangential variety $\tau_{d,n} \subset \bbP \Sym^d\bbC^{n+1}$ is strongly concise.
\end{lemma}
\begin{proof}
    Write $L = a_0x_0+\cdots + a_nx_n$ and $M = b_0x_0 + \cdots + b_nx_n$ for some $a_0,\ldots,a_n,b_0,\ldots,b_n \in \bbC$. Then, $L^{d-1} M = \sum_{\alpha} c_\alpha {\bf x}^{\alpha}$, where $\alpha = (\alpha_0,\ldots,\alpha_n)$ runs over all the multi-indices such that $\alpha_0+\cdots+\alpha_n = d$ and
    \[
        c_\alpha = \sum_{\substack{k = 0 \\ \alpha_k \neq 0}}^n b_k \binom{d-1}{\alpha_0 \cdots \alpha_{k} -1 \cdots \alpha_n}a_0^{\alpha_0} \cdots a_k^{\alpha_k -1} \cdots a_n^{\alpha_n}
    \]
    where the multinomial coefficient is $\binom{m}{m_0\cdots m_n} = \frac{m!}{m_0! \cdots m_n!}$. In other words, $c_\alpha$ is the scalar product between the vector ${\bf b} = (b_0,\ldots,b_n) \in \bbC^{n+1}$ and the vector ${\bf v}_{\alpha, {\bf a}} \in \bbC^{n+1}$ with components
    \[
        ({\bf v}_{\alpha,{\bf a}})_k = \begin{cases}
            \binom{d-1}{\alpha_0 \cdots \alpha_{k} -1 \cdots \alpha_n}a_0^{\alpha_0} \cdots a_k^{\alpha_k -1} \cdots a_n^{\alpha_n} \qquad & \text{if } \alpha_k >0, \\
            0 & \text{if } \alpha_k = 0,
        \end{cases}
    \]
    for $k = 0,\ldots, n$, where we use the notation ${\bf a} = (a_0,\ldots,a_n) \in \bbC^{n+1}$. For all $\alpha$, let $H_{\alpha, {\bf a}} \subset \bbC^{n+1}$ be the hyperplane orthogonal to ${\bf v}_{\alpha,{\bf a}} \in \bbC^{n+1}$.
    \begin{claim}
        If ${\bf a} \in \bbC^{n+1}$ is generic, then for all $\alpha$ the set $H_{\alpha,{\bf a}} \smallsetminus \bigcup_{\beta \neq a} H_{\beta, {\bf a}}$ is non-empty.
    \end{claim}
    Before proving the Claim, we see how it implies the statement. If the Claim holds, then we can choose ${\bf a} \in \bbC^{n+1}$ to be generic and, for all $\alpha$, we can choose ${\bf b}_{\alpha} \in H_{\alpha,{\bf a}} \smallsetminus \bigcup_{\beta \neq \alpha} H_{\beta, {\bf a}}$. Denote by $L, M_{\alpha} \in (\bbC^{n+1})^\vee$ the associated linear forms. Then, the coefficient of ${\bf x}^\alpha$ in $L^{d-1}M_\alpha$ is zero but all the others are non-zero.
    \begin{proof}[Proof of the Claim]
        We prove that the set of vectors ${\bf a} \in \bbC^{n+1}$ for which the Claim fails is Zariski closed. Suppose that the Claim is false. Then, there exists a multi-index $\alpha$ such that $H_{\alpha,{\bf a}} \subset \bigcup_{\beta \neq \alpha} H_{\beta, {\bf a}}$. Hence, $H_{\alpha,{\bf a}} = H_{\beta,{\bf a}}$ for some $\beta \neq \alpha$. In particular, ${\bf v}_{\alpha,{\bf a}} = {\bf v}_{\beta,{\bf a}}$, that is
        \[
            \binom{d-1}{\alpha_0 \cdots \alpha_{k} -1 \cdots \alpha_n}a_0^{\alpha_0} \cdots a_k^{\alpha_k -1} \cdots a_n^{\alpha_n} = \binom{d-1}{\beta_0 \cdots \beta_{k} -1 \cdots \beta_n}a_0^{\beta_0} \cdots a_k^{\beta_k -1} \cdots a_n^{\beta_n}
        \]
        for all $k \in \{0,\ldots,n\}$. Since $\alpha \neq \beta$ and $\alpha_0 + \cdots + \alpha_n = \beta_0 + \cdots + \beta_n = d \geq 1$, at least one of the above is a non-trivial polynomial equation vanishing on ${\bf a}$. Hence, the set of ${\bf a} \in \bbC^{n+1}$ for which the statement is false is cut out by some polynomial equations, therefore it is Zariski closed.
    \end{proof}
    This concludes the proof of the lemma.
\end{proof}

\subsubsection{Chow varieties and varieties of reducible forms}

The tangential variety to the Veronese embedding is contained, for example, in the Chow variety and the variety of reducible forms. This will be useful because of the following remark.

\begin{remark}
\label{rmk:containment}
    Let $X \subset Y$ be two algebraic varieties. If $X$ is strongly concise, then $Y$ is strongly concise. In particular, if $X$ is strongly concise, then all its secant varieties are strongly concise.
\end{remark}

Summing-up, we got the following. 

\begin{corollary}
    For each of these varieties and their secant varieties, the associated Hadamard rank is finite for any point in the ambient space.
    \begin{itemize}
        \item The Grassmannian  $\bbG(k,n) \hookrightarrow \bbP \bigwedge^{k} \bbC^{n}$ in its Plücker embedding.
        \item The tangential variety $\tau_{d,n} \subset \bbP \Sym^d\bbC^{n+1}$ to the Veronese embedding of $\bbP^n$ in degree $d$.
        \item the Chow variety
        $
            \mathbb{X}_{d,n} = \{ [L_1 \cdots L_d] : L_i \in (\bbC^{n+1})^\vee \} \subset \bbP \Sym^d\bbC^{n+1}.
        $
        \item the variety of reducible forms
        $
            \mathcal{R}_{{\bf d},n} = \{ [F_1 \cdots F_k] : F_i \in \Sym^{d_i} \bbC^{n+1}\} \subset \bbP \Sym^{d} \bbC^{n+1},
        $
        for any fixed ${\bf d} = (d_1,\ldots,d_k)$ with $d_1 + \cdots + d_k = d$.
    \end{itemize}
\end{corollary}
\begin{proof}
    The first two points are implied by \Cref{thm:maximum_finite} together with \Cref{lem:grassmannians} and \Cref{lem:tangential} respectively. Since the tangential variety to the Veronese embedding of $\bbP^n$ in degree $d$ is contained in the Chow variety $\mathbb{X}_{d,n}$ and in the variety of reducible forms $\mathcal{R}_{{\bf d},n}$ with $d_1 + \cdots + d_k = d$, the last two points are a consequence of \Cref{thm:maximum_finite}, \Cref{lem:tangential} and  \Cref{rmk:containment}.
\end{proof}

\begin{remark}
    We can rephrase the above result in terms of decompositions of tensors.
    \begin{itemize}
        \item For any $r \geq 1$, every skew-symmetric tensor can be written as a coefficient-wise product of finitely many skew-symmetric tensors of skew-symmetric rank at most $r$.
        \item For any $r \geq 1$, every homogeneous polynomial of degree $d$ can be written as a coefficient-wise product of finitely many polynomials of the form $\sum_{i=1}^r L_i^{d-1}M_i$ with $\deg(L_i) = \deg(M_i) = 1$.
        \item For any $r \geq 1$, every homogeneous polynomial of degree $d$ can be written as a coefficient-wise product of finitely many polynomials of the form $\sum_{i=1}^r L_{i,1}\cdots L_{i,d}$ with $\deg(L_{i,j}) = 1$.
        \item For any $r \geq 1$ and fixed degrees $(d_1,\ldots,d_k)$ with $d_1 + \cdots + d_k = d$, every homogeneous polynomial of degree $d$ can be written as a coefficient-wise product of finitely many polynomials of the form $\sum_{i=1}^r F_{i,1}\cdots F_{i,k}$ with $\deg(F_{i,j}) = d_j$.
    \end{itemize}
\end{remark}

\section{Border Hadamard rank}
Contrarily to the classical notion of matrix rank, it is well known that the additive notion of $X$-rank fails often to be lower semi-continuous: a first example is given by Rational Normal Curves in $\bbP^3$. 

In \Cref{example:maximum_nonfinite} it is given a family of points contained in a Hadamard power $X^{\star m}$ but with infinite Hadamard-$X$-rank. That is not the only pathology against lower semi-continuity of Hadamard ranks: in the following example, we give a point contained in $X^{\star 2}$ but with Hadamard-$X$-rank equal to $3$. 

\begin{example}\label{example:border_smallerthan_rank}
    Consider the conic $C = \{x_0(x_1+x_2) + (x_1-x_2)^2 = 0\} \subset \bbP^2$. Note that $C$ is strongly concise, thus the Hadamard rank is finite for every point by \Cref{thm:maximum_finite}. Also, note that $C$ is not binomial, thus $C^{\star 2} = \bbP^2$ by \Cref{thm:finiteness_generic}. The conic $C$ is tangent to the line $H_0 = \{x_0 = 0\}$ at the point $(0:1:1)$. We compute the Hadamard-$C$-rank of all points in $H_0$.
    \begin{itemize}
        \item $(0:1:1)$ has Hadamard-$C$-rank equal to $1$. This is the only one in $H_0$ since $C \cap H_0 = \{(0:1:1)\}$.
        \item $(0:s:t)$ with $(s:t) \not\in \{(1:1),(1:-1)\}$ have Hadamard-$C$-rank equal to $2$. Indeed, we can write 
        \[
            (0:s:t) = (0:1:1) \star (-(s-t)^2 : s(s+t) : t(s+t)). 
        \]
        \item $(0:1:-1)$ has Hadamard-$C$-rank equal to $3$. It cannot be written as an Hadamard product of two points in $C$: indeed, such a decomposition needs to involve the point $(0:1:1)$ in order to get the zero in the first coordinate, but $C$ does not contain any point of the form $(u:1:-1)$. At the same time, we have 
        \[
            (0:1:-1) = (0:1:1) \star (-1:6:3) \star(9:1:-2).
        \]
        We have already observed that $C^{\star 2}$ fills the ambient space, thus $(0:1:-1)$ is a limit of a $1$-parameter family of points with Hadamard-$C$-rank equal to $2$. More explicitly, we can write
        \[
            (0:1:-1) = \lim_{\epsilon \rightarrow 0} ~(0:1+\epsilon:-1+\epsilon) = \lim_{\epsilon \rightarrow 0} ~(0:1:1) \star (-2:\epsilon(1+\epsilon):\epsilon(-1+\epsilon)).
        \]
    \end{itemize}
\end{example}

The latter example justifies the following definition.

\begin{definition}\label{def:hadamard_borderrank}
    Let $X \subset \bbP V$ be an algebraic variety and $p\in \bbP V$. The {\em border Hadamard-}$X${\em -rank} of $p$ is 
    \[
        \underline{\Hrk}_X(p) = \min \{ m \ | \ p \in X^{\star m} \}.
    \]
    with the convention that if $p \not \in X^{\star m}$ for all $m \geq 1$, then $\underline{\Hrk}_X(p) = \infty$.
\end{definition}

We now show that for varieties whose points do not contain too many zeroes, border rank equals rank in the low-rank regime. 

We recall the following notation from \cite{bocci2024hadamard}.

\begin{notation}
    For all $i = 0,\ldots, N$, we denote by $\Delta_i \subset \bbP^N$ the variety of points with at most $i+1$ non-zero coordinates. Hence
    \[
        \Delta_0 = \{ (1:0:\cdots:0), \ldots, (0:\cdots:0:1)\} \subset \Delta_1 \subset \cdots \subset \Delta_{N-1} = H_0 \cup \cdots \cup H_N \subset \Delta_N = \bbP^N.
    \]
    In particular, $\Delta_{N-k} \subset \bbP^N$ is the variety of points having at least $k$ coordinates equal to zero.
\end{notation}

\begin{remark}
    Let $X \subset \bbP^N$ be an irreducible variety of dimension $n$. Since a linear space of dimension at least $N - n$ intersects $X$, one has that $X \cap \Delta_i \neq \emptyset$ for all $i \geq N-n$. If $n>N/2$, then $X \cap \Delta_{N-n-1} \neq \emptyset$.
\end{remark}

\begin{remark}
\label{rmk:curve}
    A curve $C \subset \bbP^N$ with $C \cap \Delta_{N-2} = \emptyset$ is strongly concise because, by definition, $C \cap (H_i \cap H_j) = \emptyset$ for all $i\neq j$.
\end{remark}
We begin here to investigate the difference between rank and border rank.
\begin{lemma}
\label{lemma:borderequalsrank}
    Let $X \subset \bbP^N$ be an irreducible variety of dimension $n\leq N/2$ such that $X \cap \Delta_{N-n-1} = \emptyset$. Then, for all $r \geq 2$ such that $rn\leq N$, the points $p \in \eta_r(X) \smallsetminus \eta_{r-1}(X)$ satisfy $\underline{\Hrk}_X(p) = \Hrk_X(p) = r$.
\end{lemma}
\begin{proof}
    Since $X \cap \Delta_{N-n-1} = \emptyset$, the points of $X$ have at most $n$ coordinates equal to zero. A product of $r$ points in $X$ has at most $rn$ coordinates equal to zero: thus, by numerical assumption $rn \leq N$, the restriction $h\colon X^r \dashrightarrow X^{\star r}$ of the Hadamard map to $X^r$ is a morphism. Since morphisms of projective varieties are closed, we conclude that
    $X^{\star r} = \{ p_1 \star \cdots \star p_r : p_i \in X \}.$
    This implies the statement.
\end{proof}

\begin{theorem}
\label{thm:curves_finite}
    Let $C \subset \bbP^N$ be an irreducible curve such that $C \cap \Delta_{N-2} = \emptyset$. Then:
    \begin{enumerate}
        \item for all $1 \leq r \leq N$, $\eta_r(C)$ has dimension $r$;
        \item for all $p \in \bbP^N$, $\underline{\Hrk}_C(p) = \Hrk_C(p) \leq N$ and equality holds when $p$ is generic.
    \end{enumerate}
\end{theorem}
\begin{proof}
    Since (1) implies that $\eta_N(C) = \bbP^N$, (2) follows from (1) and \Cref{lemma:borderequalsrank}. We prove (1).

    By assumption, $C$ contains a point $p$ with all non-zero coordinates, thus, without loss of generalities, we can assume that $(1:\cdots:1) \in C$. Indeed, let $q = p^{\star (-1)}$ and consider $q \star X$. Clearly, $(1:\cdots:1) = q \star p \in q \star X$. Moreover, by \cite[Lemma~3.13]{hadamardranks}, $(q \star X)^{\star r} = q^{\star r} \star X^{\star r}$: in particular, the dimensions of all Hadamard powers of $X$, thus of $\eta_r(X)$, are preserved by the diagonal change of coordinates given by the Hadamard product with $q$. Thus, it is enough to prove (1) under the assumption $(1:\cdots:1) \in C$. 
    
    As mentioned in \Cref{sec:definitions}, if $(1:\cdots:1) \in C$, then $\eta_r(C) = C^{\star r}$. Suppose by contradiction that there exists $2 \leq s \leq N$ such that $\dim C^{\star s} \leq s-1$ but $\dim C^{\star r} = r$ for $r \leq s-1$. Then, since $C^{\star (s-1)} \subset C^{\star s}$, we have that $s-1 \leq \dim C^{\star s} \leq s-1$, from which $\dim C^{\star s} = s-1 = \dim C^{\star (s-1)}$. Since Hadamard powers are irreducible (\Cref{rmk:hadamardirreducible}), we have $C^{\star s} = C^{\star (s-1)}$. This implies that $C^{\star (s+i)} = C^{\star (s-1)}$ for all $i \geq 0$: in particular, $C^{\star m} \neq \bbP^N$ for all $m\geq 1$, hence $\Hrk_C^\circ = \infty$. This is a contradiction: indeed, by \Cref{rmk:curve}, $C$ is strongly concise and, in particular, it has finite generic Hadamard rank by \Cref{thm:maximum_finite}.
\end{proof}

\begin{remark}
    For a given irreducible curve $C \subset \bbP^N$ and a general linear transformation $g \in GL_N$, it holds $g(C) \cap \Delta_{N-2} = \emptyset$, hence \Cref{thm:curves_finite} applies to $g(C)$. 
\end{remark}

\begin{remark}
    Note that the assumption $X \cap \Delta_{N-n-1} = \emptyset$ in \Cref{lemma:borderequalsrank}, and thus the assumption $C\cap \Delta_{N-2} = \emptyset$ in \Cref{thm:curves_finite}, is needed: indeed, \Cref{example:border_smallerthan_rank} provides a counterexample to the equality between rank and border rank since the curve $C \subset \bbP^2$ contains the coordinate point $(1:0:0)$. 
\end{remark}

The bound in \Cref{prop:bound_zeroes} gives as a consequence that if $X$ is strongly concise, then $\Hrk_X^{\max}\leq 3N$. However, if the variety $X$ has not too many zeroes, we can show that $\Hrk_X^{\max}\leq N$. This is the content of \Cref{thm:maximum_notmanyzeroes}, of which we now
give a proof.

\begin{proof}[Proof of \Cref{thm:maximum_notmanyzeroes}]
    The case $n = 1$ is \Cref{thm:curves_finite}. If $n \geq 2$, let $L \subset \bbP^N$ be a general linear space of codimension $n-1$. By the theorem of Bertini (see \cite[Theorem 3.3.1]{lazarsfeldpositivity} or \cite[Section I.6]{jouanolou1983theoremes}), $C = X \cap L$ is an irreducible curve and $\Hrk_X(p) \leq \Hrk_C(p)$ for all $p \in \bbP^N$. 
    
    We notice that $C \cap \Delta_{N-2} = \emptyset$. In fact, suppose by contradiction that $C \cap (H_i \cap H_j) = (X \cap L) \cap (H_i \cap H_j) \neq \emptyset$ for $i \neq j$. We recall that $\Delta_{N-n-1} = \bigcup_{k_1,\ldots, k_{n+1}} (H_{k_1} \cap \cdots \cap H_{k_{n+1}})$. Since $L$ is generic, by semicontinuity:
    \[
        \dim( X \cap H_i \cap H_j \cap H_{k_1} \cap \cdots \cap H_{k_{n-1}}) \geq \dim(X \cap H_i \cap H_j \cap L) = \dim (C \cap H_i \cap H_j) \geq 0
    \]
    for some distinct $i,j,k_1,\ldots,k_{n-1}$, hence $X \cap \Delta_{N - n -1} \neq \emptyset$, a contradiction.
    
    In conclusion, \Cref{thm:curves_finite} applies to $C$, and then $\Hrk_X(p) \leq \Hrk_C(p) \leq N$ for all $p \in \bbP^N$.
\end{proof}

We notice that the bound in \Cref{thm:maximum_notmanyzeroes} is sharp.

\begin{example}\label{example:sharp}
Let $C\subset \bbP^N$ be an irreducible curve such that $C\cap \Delta_{N-2}=\emptyset$. Let $p = (1:0:\cdots :0)$ or any other coordinate point. By \Cref{thm:curves_finite}, we have that $\Hrk_C(p) \leq N$. However, since each point of $C$ has at most one coordinate equal to zero, $\Hrk_C(p) \geq N$, hence $\Hrk_C(p) = N$.
\end{example}

\bibliographystyle{alphaurl}
\bibliography{hadarank_finiteness.bib}

\end{document}